\documentclass[11pt,reqno]{amsart}
\usepackage{amsbsy,amsfonts,amsmath,amssymb,amscd,amsthm,mathrsfs}
\usepackage{subfigure,enumitem,graphicx,color}

\theoremstyle{plain}
\newtheorem{mainthm}{Theorem}

\newtheorem{theorem}{Theorem}[section]
\newtheorem{lemma}[theorem]{Lemma}

\newtheorem{corollary}[theorem]{Corollary}
\newtheorem{conjecture}{Conjecture}
\newtheorem{definition}{Definition}

\theoremstyle{definition}
\newtheorem{example}[theorem]{Example}

\newtheorem{problem}[theorem]{Problem}

\DeclareMathOperator{\IFS}{IFS}
\DeclareMathOperator{\vol}{Vol}

\begin{document}

\bibliographystyle{amsplain}

\title[Non-removable term ergodic action semigroups/groups]{Non-removable term ergodic action semigroups/groups}
\author[A. Sarizadeh]{Ali Sarizadeh}
\address{Department of Mathematics,
 Ilam University, P.O.Box 69315-516, Ilam, Iran.}
\email{ali.sarizadeh@gmail.com}

\keywords{action semigroup, action group,
minimality, ergodicity, robust property, circle packing}.\\
\subjclass[2000]{26A18, 37A99, 28A20}
\begin{abstract}

\noindent  In this work, we introduce the concept of term  ergodicity for action semigroups and construct
semigroups on two dimensional manifolds  which are $C^{1+\alpha}$-robustly term ergodic. Moreover,
we illustrate the term ergodicity by some exciting examples.

At last, we study a problem in the context of circle packing which is concerned to term ergodic.
\end{abstract}

\maketitle
\section*{Introduction}
As is well known, in ordinary dynamical system $(M,f)$, where $M$ is
a compact metric space and $f$ is a map from $M$ to itself, an
invariant probability measure $\mu$ is said to be ergodic if every
invariant measurable set is either of zero or full $\mu$-measure.
The definition of ergodicity, for action semigroups/groups
generated by $\mathcal{G}=\{g_{1},\ldots,g_{s}\}$ on $M$, is
naturally extended to quasi-invariant measure\footnote{A measure
$\mu$ is said to be quasi-invariant if $(g_i)_\ast\mu$ is absolutely
continuous with respect to $\mu$ for every element $g_i$ in
$\mathcal{G}$.}. Recall that a quasi-invariant measure is ergodic if
every measurable invariant set with respect to $\mathcal{G}$ is
either of zero or full measure. In view of topological theory,
counterpart to ergodicity is minimality. More precisely, an action semigroup/group is said to be minimal if each closed invariant
subset is either empty or coincides with the  whole of space.
Minimal systems have been studied extensively by many authors, see
for instance ~\cite{dkn,ghs,hn,1002.37017,1086.37026,br}.
Authors in ~\cite{1002.37017}, provided an
example of an action semigroup/group generated by two circle
diffeomorphisms, that is robustly minimal in the
$C^1$-topology.
In ~\cite{ghs}, this one-dimensional example is
generalized to an action semigroup/group on higher-dimensional
manifolds which is also $C^1$-robustly minimal. Recently, in ~\cite{hn},
the authors  provided an action semigroup generated by two
diffeomorphisms on any
compact manifold that is  $C^1$-robustly minimal.
In typical papers, finding a local invariant set with non-empty interior plays a key role.

Now, we are going to concentrate on some relations between
ergodicity and  minimality in action semigroups/groups. Notice,
there are some examples of ergodic action semigroups/groups having
global fixed points. So, in general, ergodicity does not imply
minimality. Thus, in the opposite direction, a natural question
arises: which system having the minimal property is ergodic? To
answer this question one can refer to an earlier result contained
in~\cite{dkn} which allows to solve the following conjecture
concerning the one-dimensional case in the affirmative under some
additional assumptions, although the conjecture and the question are
apparently far from each other.
\begin{conjecture}\label{con}
Every minimal smooth action of a finitely generated group on the
circle is ergodic with respect to the Lebesgue measure.
\end{conjecture}

In this note, we prove the above mentioned conjecture under some assumption which is not unusual.
Actually, by this assumption, we insure that our result does not have any contradiction with results by Furstenberg \cite{fu}.

On the other hand, if there exists a relation between
minimality and ergodicity, it is natural to get, as a
corollary, robustness of ergodic systems. By these results, we
provide an example of some action semigroup/group which is both of
robustly minimal and robustly term ergodic.

Finally, as an application of this note, we study the branch of mathematics generally known as \emph{circle packing}. The packing problem is concerned with how to pack a number of objects, each with given shape and size, into bounded region  without overlap (see more details on this context \cite{gjls}, also see http://hydra.nat.uni-magdeburg.de/packing/cci/\#Overview). Here, we consider a problem of packing of circles with unequal radii and some additional property into a given circle. In fact, by benefit the concept of term ergodicity, we this problem in the circle packing.

\section{Main results}
We begin to introduce some definitions and notations and then formulate our main results. Throughout
this paper, $M$ stands for a smooth compact Riemannian manifold and
$\vol$ is normalized volume. Also, consider the space
$\mathrm{Diff}^1(M)$ of $C^1$ diffeomorphisms of $M$, endowed with
the $C^1$-topology. A point $x\in M$ is a Lebesgue density point of
measurable set $A\subseteq M$ if
$$\lim_{r\rightarrow 0}\vol(A:B(x,r))=\frac{\vol(A\cap
B(x,r))}{\vol(B(x,r))}=1,$$ where $B(x,r)$ is the geodesic ball of
radius $r$ centered about $x$. Denote by $DP(A)$ the set of Lebesgue
density points of a measurable set $A$. By Lebesgue density point
theorem, for every measureable set $A$, $$\vol(A\bigtriangleup
DP(A))=0.$$

Now, consider a collection of diffeomorphisms
$\mathcal{G}=\{g_{1},\ldots,g_{s}\}$ on $M$. Write
$\mathcal{G}^{-1}=\{g_{1}^{-1},\ldots,g_{s}^{-1}\}$. The action semigroup
$<\mathcal{G}>^+$ generated by $\mathcal{G}$ is given by
\[
  <\mathcal{G}>^+= \{h:M\to M \colon \, h=g_{i_n}\circ \dots \circ
  g_{i_1}, \ i_j \in \{1,\ldots,s\}  \}\cup \{ \mathrm{id}:M\to M \}.
\]
Furthermore,  the action group
$<\mathcal{G}>$ generated by $\mathcal{G}$ is defined by $<\mathcal{G}\bigcup \mathcal{G}^{-1}>^+$.
So, every action group is not more than an action semigroup.
Notice that $\lim_{r\to \infty} f^r_{\omega}\not\in<\mathcal{G}>^+$ where $f^r_{\omega_1,\dots,\omega_r}=f_{\omega_n}\circ f^{r-1}_{\omega_1,\dots,\omega_{r-1}}$. Also, we denote the reverse iteration by $\hat{f}^r_{\omega_1,\dots,\omega_r}=f_{\omega_1}\circ f^{r-1}_{\omega_2,\dots,\omega_r}$.

Let us recall that the subset $\mathcal{K}$ of $M$ is invariant with respect to an action semigroup
$\Gamma$ generated by $\mathcal{G}=\{g_{1},\ldots,g_{s}\}$ if
$$\mathcal{K} =\bigcup_{i=1}^sg_i(\mathcal{K}).$$
Also, $\Gamma$ is said to be \emph{minimal} if each closed invariant subset $A$ of $M$ with respect to $\Gamma$ is empty or coincides with $M$.

Observe that for ordinary dynamical system $(M,f)$, the minimality
of $f$ is equivalent to that of $f^{-1}$. This is not the case for
dynamical systems with several maps: there exists a minimal action semigroup
$<f_1,\dots, f_s>^+$ on the circle such that
$<f_1^{-1},\dots,f_s^{-1}>^+$ is not
minimal~\cite{1086.37026}.
\begin{definition}
Given an action semigroup  $\Gamma$ generated by $\mathcal{G}=\{g_{1},\ldots,g_{s}\}$ and a probability measure space $(M,\mathcal{M},\mu)$ which $\mu$ is quasi-invariant with respect to $\Gamma$. The measure $\mu$ is called \emph{ergodic } if for
every measurable set $B\in \mathcal{M}$ with
$$g^{-1}_i(B)=B;\ \ \ \forall \ i=1,\dots,s\ $$
we have that either $\mu(B)=0$ or $\mu(B)=1$.
\end{definition}
Here, to obtain term ergodic results for an action semigroup/group, we
need   just $C^{1+\alpha}$-regularity. In this regard, we begin by stating term ergodic result for an action semigroup.
\begin{mainthm}\label{thmA}Every boundaryless compact two dimensional manifold $M$ admits a finite set of $C^{1+\alpha}$-diffeomorphisms that generates a $C^{1+\alpha}$-robustly term ergodic action semigroup with respect to volume measure.
\end{mainthm}
We recall that a property $\mathcal{P}$ holds $C^r$-robustly for action semigroup $\Gamma$ generated by $\mathcal{G}=\{g_{1},\ldots,g_{s}\}$ if it holds for action semigroup $\hat\Gamma$ generated by $\mathcal{F}=\{f_{1},\ldots,f_{s}\}$ whose elements are $C^r$-perturbations of elements of $\mathcal{G}$.

Since every action group is action semigroup the following corollary is an immediately consequence of Theorem \ref{thmA}.
\begin{corollary}\label{thmB}
Theorem \ref{thmA} is valid for action groups.
\end{corollary}

\section{Some new results about minimality}
First of all, we state some results about minimality on compact manifolds in any dimension.
\begin{lemma}\label{lemm1}
Let $M$ be a compact manifold and action semigroup generated by homeomorphisms $\mathcal{G}=\{g_{1},\ldots,g_{s}\}$ be a minimal. Then every invariance set respect to each $g_i$ for $i=1,\dots,s$ and its complement are dens in the whole space.
\end{lemma}
\begin{proof}
Suppose that
$$g_i(\mathcal{B})=\mathcal{B};\ \ \ \ \forall\ i=1,\dots,s.$$
So we have
$$g_i(\overline{\mathcal{B}})=\overline{\mathcal{B}};\ \ \ \ \forall\ i=1,\dots,s.$$
When $\overline{\mathcal{B}}$ is closed and $<\mathcal{G}>^+$
is minimal on $M$ then $\overline{\mathcal{B}}=M$. On the other hand
 $g_i(\mathcal{B}^c)= \mathcal{B}^c$
for $i=1,\dots,s$, where $\mathcal{B}^c$ is complement of
$\mathcal{B}$. By minimality of $<\mathcal{G}>^+$ and
invariance of the subset $\mathcal{B}^c$ with respect to each
generator, one can have density of this subset in $M$; that is
$\overline{\mathcal{B}^c}=M.$
\end{proof}
It follows that
$$\mathcal{B}\cap B(x,r)\neq\emptyset\ \ \ \mathrm{and} \ \ \ \mathcal{B}^c\cap B(x,r)\neq\emptyset$$
for each real number  $r>0$ and $x\in M$.
\begin{lemma}\label{lemm2}
Under the assumption of lemma \ref{lemm1} on a compact Riemannian manifold $M$, density points of every invariance set respect to each $g_i$ for $i=1,\dots,s$, with positive volume is dens in the whole space.
\end{lemma}
\begin{proof}
Suppose that
$0<\vol(\mathcal{B})$ and
$$g_i(\mathcal{B})=\mathcal{B};\ \ \ \ \forall\ i=1,\dots,s.$$
If $DP(\mathcal{B})$ is not dense then there exists a neighborhood
$B(x_0,r_0)$ for some point $x_0$ of $M$ so that $B(x_0,r_0)\cap
DP(\mathcal{B})=\emptyset $. By Lebesgue density point theorem, one
can have $\vol(B(x_0,r_0)\cap \mathcal{B})=0$, equivalently
$\vol(B(x_0,r_0)\cap\mathcal{B}^c)=\vol(B(x_0,r_0))$.
We remark that the volume measure is a quasi-invariant for any $C^1$-diffeomorphism. So, the
union of iterates $B(x_0,r_0)\cap
\mathcal{B}$ under $\mathcal{G}$  has zero volume measure.
Now,  the assumption of minimality $<\mathcal{G}>^+$ and invariance of $\mathcal{B}^c$ under $\mathcal{G}$ yields a contradiction with $0<\vol(\mathcal{B})$. Hence
$\overline{DP(\mathcal{B})}=M$.
\end{proof}
Notice that, similarly, when $0<\vol(\mathcal{B})<1$ both the subsets $DP(
\mathcal{B})$ and $DP(
\mathcal{B}^c)$ are dense.

\begin{lemma}\label{minimal}
Let $M$  be a compact connected two dimensional manifold. Then there exist a set $\mathcal{H}=\{h_1,\dots,h_k\}$ of
$C^1$-diffeomorphisms on $M$ and invariant set $\Delta$ with respect to $\mathcal{H}$ and nonempty interior  so that
$<\mathcal{H}>^+$ is minimal on $\Delta$ and for every $x\in \Delta$ and $i=1,\dots,k$, $Dh_i$ at $x$ have two complex eigenvalues.
\end{lemma}
\begin{proof}
Let $A$ be rotation matrix by angle $\theta=179^\circ$. Define a linear map $T$ as the follows,
$$
T(x,y)=\kappa\cdot A(x,y);\ \ \ \text{for ever} (x,y)\in\mathbb{R}^2,
$$
where $3/4<\kappa<1$. The choice  of $\theta$ insure that each eigenvalues of $T$ at each point is complex.  Put $V=B(0,\delta)$. Clearly, $T(\overline{V})\subset V$ and for every $x\in \overline{V}$, $DT$ at $x$ have two complex eigenvalues $\lambda,\overline{\lambda}$ with $|\lambda|=|\overline{\lambda}|>3/4$.
For every $y\in W$, define
$$
S_y(x)=T(x-y)+y:\ \ \forall\ x\in V
$$
where $W=\{x\in V:\ |x|=\frac{3}{4}\delta\}$. Observe that, by construction of $T$ and $S_y$, the following set
$$
T(V)\bigcup\ [\bigcup_{y\in W}S_y(V)]
$$
is a cover for $\overline{V}$.
Therefor, it has a finite subcover as $\overline{V}\subset T(V) \bigcup\ [\bigcup_{i=1}^{k-1} S_i(V)].$ Take $\mathcal{H}_0=\{T,S_1,\dots,S_{k-1}\}$.
Since every element of $\mathcal{H}_0$  is contractions, there is a ball $U$ that is mapped into itself by $T$ and $S_1,\dots,S_{k-1}$, i.e., $\mathcal{H}_0(U)=T(U) \bigcup\ [\bigcup_{i=1}^{k-1} S_i(U)]\subset U$ (when $\kappa$ is very close to $3/4$ one can take $U=B(0,16\delta)$ to have this property). Thus
$$
\lim_{n\to \infty}\mathcal{H}_0^n(U)=\lim_{n\to \infty}\mathcal{H}_0^{n-1}[\mathcal{H}_0(U)]=\Delta
$$
is  the unique non-empty invariant set with respect to $\mathcal{H}_0$ which the interior of it is non-empty. Notice that the semigroup generated by $\mathcal{H}_0$ on $\Delta$ is minimal.

Take a gradient Morse-Smale vector fields $\dot{x}= \nabla F_i(x)$ on $M$ with a unique hyperbolic attracting equilibrium $p_i$ , for $i=1,\dots,k$ (see e.g. Theorem 3.35 of \cite{m},  for the
existence of Morse functions $F$ with unique extrema) and let $h_i$ be its time-1 map. We may assume that
each $p_i$ belong to an open neighborhood  $\hat{V}$ and the each eigenvalue of $Dh_i(p_i)$ are complex.

Now, working in a coordinate chart on a small open neighborhood $\hat{V}\subset\mathbb{R}^2$ and $U\subset \hat{V}$.
One can assume that $h_1=T$ and $h_i=S_i$ on $U$ for $i=2,\dots,k$. The action semigroup generated by
$\mathcal{H}=\{h_1,\dots,h_k\}$ is minimal on the set $\Delta$.
\end{proof}

Since the construction used in Lemma~\ref{minimal} is $C^1$-robust, by the similar argument used in \cite{ghs} and \cite{hn}, we can have  a finite extension $\mathcal{G}$ of $\mathcal{H}$ so that
$<\mathcal{G}>^+$ is $C^1$-robust minimal on $M$.
\begin{corollary}\label{cormin}
There is  finite extension of  $\mathcal{H}$ to finite set $\mathcal{G}$ which $<\mathcal{G}>^+$ is $C^1$- robustly minimal.
\end{corollary}
\section{Robustly term ergodic: Proof of Theorem \ref{thmA}}
We used the following lemma to prove of main theorem.
\begin{lemma}[Bounded distortion in the Hutchinson attractor]
\label{prop-dist}
Consider a finite family
$
\mathcal{H}=\{h_1,\ldots,h_k\} \subset \mathrm{Hom}(M)
$
where each $h_i$ is a contracting $C^{1+\alpha}$-map of the closure of an open set $D \subset M$. Then, there exists $L_{\mathcal{H}}>0$ such that for every
$n\in\mathbb{N}$ and $\omega\in \Sigma_k^+$,
$$
L_{\mathcal{H}}^{-1}<\bigg|\frac{\det(D\hat{h}^n_\omega(x))}{\det(D\hat{h}^n_\omega(y))}\bigg|<L_{\mathcal{H}} \quad \text{for all $x,y\in \triangle_{\mathcal{H}}$}
$$
where $\triangle_\mathcal{H}$ is the Hutchinson attractor of $\IFS(\mathcal{H})$ in $\overline{D}$.
\end{lemma}
\begin{proof}
Define $\Phi:\textrm{GL}(\textrm{dim}(M),\mathbb{R})\to\mathbb{R}$
by $\Phi(A)=\log|\det(A)|$ and $\digamma_i(x)=\Phi(Dh_{i}(x))$, for
any $i=1,\ldots,k$. Note that
by assumption, $\log |\det Dh_i|$ is $\alpha$-H\"older
and thus for every
$x,y\in\triangle_{\mathcal{H}}$ and $1\leq i\leq k$
\[|\digamma_i(x)-\digamma_i(y)|\leq C\|x-y\|^\alpha,\]
for some constants $C>0$. On the other hand, for every $\omega=\omega_1\omega_2\dots\in\Sigma_k^+$,
\[\|\hat{h}^n_\omega(x)-\hat{h}^n_\omega(y)\|\leq
\|Dh_{\omega_1}\|~\|\hat{h}^{n-1}_{\omega_2,\dots,\omega_n}(x)-\hat{h}^{n-1}_{\omega_2,\dots,\omega_n}(y)\|\leq
\xi^n\|x-y\|\leq \xi^n\textrm{diam}(\triangle_{\mathcal{H}}),\] where
$$
\xi=\sup_{x\in \triangle_\mathcal{H},\ 1\leq i\leq
k}\|Dh_{i}(x)\|<1.
$$
Hence,
\begin{align*}
\log\frac{|\det\big(D\hat{h}^n_\omega(x)\big)|}{|\det\big(D\hat{h}^n_\omega(y)\big)|}
&=\sum_{i=0}^{n-1}|\digamma_{\omega_i}(\hat{h}^i_\omega(x))-\digamma_{\omega_i}(\hat{h}^i_\omega(y))|
\leq C \sum_{i=0}^{n-1}
\|\hat{h}^i_\omega(x)-\hat{h}^i_\omega(y)\|^\alpha \\ &\leq C
\sum_{i=0}^{n-1}\{\xi^i\textrm{diam}(\triangle_{\mathcal{H}})\}^\alpha
\leq CM
(\textrm{diam}(\triangle_{\mathcal{H}}))^\alpha\sum_{i=0}^{\infty}(\xi^\alpha)^i.
\end{align*}
Taking $L_{\mathcal{H}}=\exp\{C \xi^\alpha
(\textrm{diam}(\triangle_{\mathcal{H}}))^\alpha/ (1-\xi^\alpha) \}$
the desired inequality holds.
\end{proof}
Now, we will prove Theorem \ref{thmA}.
\begin{proof}[Proof of Theorem~\ref{thmA}]
Suppose that $M$ is two dimensional manifold. Consider Lemma \ref{minimal} for $C^{1+\alpha}$-diffeomorphisms  $\mathcal{H}=\{h_1,\dots,h_k\}$, which will provide an invariant set
$\Delta$ with nonempty interior respect to action semigroup $<\mathcal{H}>^+$
so that the $Dh_i(x)$ have complex eigenvalue for each $i=1,\dots,k$ and $x\in \Delta^\circ$.
Take the subset $U$ of $M$ so that is mapped into itself by $h_i$, i.e., $\mathcal{H}(U)=\bigcup_{i=1}^{k} h_i(U)\subset U$ and  also take $<\mathcal{G}>^+$ is a finite $C^{1+\alpha}$-extension of $\mathcal{H}$ which is $C^{1+\alpha}$-robustly minimal on $M$ (see Corollary \ref{cormin}).

We claim that $<\mathcal{G}^{-1}>^+$ is term ergodic. To this end, suppose that $0<\vol(\mathcal{B})<1$  and
$$g_i(\mathcal{B})=\mathcal{B};\ \ \ \ \forall\ i=1,\dots,s.$$

Now, let $p\in \Delta^\circ$ and  $J=B(p,\hat{\delta})\subset \Delta^\circ$.
Since  $Dh_i(x)$ have two complex eigenvalue $\lambda_i,\overline{\lambda_i}$ for every $x\in U$ and $i=1,\dots,k$,
the image of an open ball under the apply of $h_i$ is a ball, too.

On the other hand, there is
$\omega\in \Sigma^+_k$ so that $\bigcap_r\hat{h}^r_\omega(U)=\{p\}$ and $\lim_{r\to\infty}\text{diam}(\hat{h}^r_\omega(U))\to 0$.
So, $\hat{h}^r_\omega(U) $ is a ball for each $r$ and  $\text{diam}(\hat{h}^r_\omega(U)\to 0$ as $r\to \infty$. Define
$$
r_0=\min\{r:\ \text{diam}(\hat{h}^r_\omega(U))<\delta\}.
$$
Clearly, $p\in \hat{h}^{r_0}_\omega(U)\subset B(p,\delta)$. Since $\hat{h}^{r_0}_\omega(U)$ is a ball with  $\text{diam}(\hat{h}^{r_0}_\omega(U))<\delta$ and $\text{diam}(\hat{h}^{r_0-1}_\omega(U))\geq\delta$, then
$$
\text{diam}(h_i(B(p,\delta/2)))\leq\text{diam}(\hat{h}^{r_0}_\omega(U))<\delta
$$
and
$$
\frac{\vol(h_i(B(p,\delta/2)))}{\vol(J)} \leq\frac{ \vol(h_{\omega}^{r_0}(U))}{\vol(J)}.
$$
Notice that  $h^{r_0}_\omega(U)\subset J$ and $h_i(\mathcal{B}^c)=\mathcal{B}^c$ for every $i=1,\dots, k$ when $\mathcal{H}$ is a subset of $\mathcal{G}$. So, one can have $h_\omega^{r_0}(\mathcal{B}^c\bigcap U)\subseteq \mathcal{B}^c\bigcap h_\omega^{r_0}(U)$.

Hence, we have
\begin{align*}
\frac{\vol(\mathcal{B}^c\bigcap J)}{\vol(J)}&\geq
\frac{\vol(\mathcal{B}^c\bigcap \hat{h}_{\omega}^{r_0}(U))}{\vol(J)}\\
&\geq \frac{\vol(\mathcal{B}^c\bigcap \hat{h}_{\omega}^{r_0}(U))}{\vol(\hat{h}_{\omega}^{r_0}(U))}.
\frac{\vol(\hat{h}_{\omega}^{r_0}(U))}{\vol(J)}\\
&> \frac{\vol(\hat{h}_{\omega}^{r_0}(\mathcal{B}^c\bigcap U))}{\vol(\hat{h}_{\omega}^{r_0}(U))}\cdot\frac{\vol(h_i(B(p,\delta/2)))}{\vol(J)}.
\end{align*}
On the other hand  $\mathcal{B}^c$ is forward $\mathcal{H}$-invariant,
\begin{align*}
\frac{\vol(\hat{h}^{r_0}_\omega(\mathcal{B}^c\cap U))}{\vol(\hat{h}_{\omega}^{r_0}(U))} \geq L_\mathcal{H}^{-1}\, \frac{\vol(\mathcal{B}^c\cap U)}{\vol(U)}
\end{align*}
Indeed, the last inequality is implied
by the bounded distortion result,
Lemma~\ref{prop-dist}. Indeed, since
\begin{align*}
\vol(\hat{h}^{r_0}_\omega(\mathcal{B}^c\cap U))&=\int_{\mathcal{B}^c\cap U}
|\det D\hat{h}^{r_0}_\omega | \,d\vol\geq
 \vol(\mathcal{B}^c\cap U)
\inf_{x\in \Delta} |\det D\hat{h}^{r_0}_\omega(x)| , \\
\vol(\hat{h}^{r_0}_\omega(U))&=\int_{U}
|\det \hat{h}^{r_0}_\omega|\,d\vol\leq  \vol(U) \sup_{x\in \Delta}|\det D\hat{h}^{r_0}_\omega(x)|
\end{align*}
This  means that $\frac{\vol(\mathcal{B}^c\bigcap J)}{\vol(J)}$ is bounded from below for every neighborhood $J$ of $p$. So, $p\not\in DP(\mathcal{B})$. Similarly, One can prove that $p\not\in DP(\mathcal{B}^c)$.
Hence $$\Delta^\circ\bigcap(DP(\mathcal{B})\bigcup DP(\mathcal{B}^c))=\emptyset$$
which is a contradiction with Lemma \ref{lemm2} and the proof is completed.
\end{proof}
Observe that ergodicity may be removed from some ordinary dynamical
systems under a perturbation which an irrational translation is such
a system. But the following example, containing irrational
translation, is robust term ergodic. Moreover, it is shown that
sufficiently close to such system in the $C^{1+\alpha}$-topology, there is a
term ergodic action semigroup which each of generators are
not ergodic.
\begin{example}Suppose
$f_1$ is a north-south pole $C^{2}$-diffeomorphism of the circle
$\mathbb{S}^1$ possessing an attracting fixed point $p$ as a north
pole and a repelling fixed point $q$ as south pole with multipliers
$$1/2<f_1^\prime(p)<1\ \ \ \ and\ \  \ \ 1/2<(f_1^{-1})^\prime(q)<1.$$
Consider the map $f_2=R_\lambda$ where $R_\lambda$ is the rotation
by irrational angle $\lambda$ on the circle. Observe that, both
systems $<f_1,f_2>^+$ and $<f_1^{-1},f_2^{-1}>^+$  are $C^{1+\alpha}$-
robust minimal~\cite{1002.37017}. Thus the systems $<f_1,f_2>^+$
and $<f_1^{-1},f_2^{-1}>^+$ are $C^{1+\alpha}$-robust term ergodic.
 Moreover, take a rational number $\gamma$
sufficiently close to $\lambda$ so that the system
$<f_1,R_{\gamma}>^+$ is minimal, too. Clearly, none of
$f_1,f_1^{-1},R_{\gamma},R_{\gamma}^{-1}$ is neither ergodic nor minimal but,
$<f_1,R_{\gamma}>^+$ is both of minimal and term ergodic.
\end{example}
\section{A problem on circle packing concerning to ergodicity }
Let $M$ be a compact $2$-dimensional manifold and $\Gamma=<\mathcal{G}\bigcup\mathcal{G}^{-1}>^+$ be a minimal action group generated by homeomorphisms $\mathcal{G}=\{g_{1},\ldots,g_{s}\}$.
Suppose that $\mathcal{B}$ is an invariant set with respect to each generator of $\Gamma$, that is
$$g_i(\mathcal{B})=\mathcal{B};\ \ \ \ \forall\ i=1,\dots,s.$$
Assume that $0<\vol(\mathcal{B})<1$. By Lemmas \ref{lemm1} and \ref{lemm2}, we insure that $\mathcal{B},\ \mathcal{B}^c,\ DP(\mathcal{B}),\ DP(\mathcal{B}^c)$ are dense in $M$.

Now, let $y$ be an arbitrary point of $DP(\mathcal{B})$. By
definition of density point, one can find $\delta>0$ so that
$$\vol(\mathcal{B}\bigcap B(y,\delta))>\frac{3}{4}\vol(B(y,\delta)).$$
One  kind of circle packing problem, with respect to dynamical system, may be as follows.
\begin{problem}
Under the above assumptions, is there a family
$\{B(p,\delta_p)\}_{p\in \mathcal{P}}$ with  $\mathcal{P}\subset
DP(\mathcal{B}^c)\bigcap B(y,\delta)$ so that
\begin{enumerate}[label=(\roman*),ref=\roman*]
\item\label{it:1} $B(p,\delta_p)\subseteq B(y,\delta)$,
\item\label{it:2} $\forall\ p,q\in \mathcal{P}$ with $p\neq q;\ \ \ B(p,\delta_p)\cap
B(q,\delta_q)=\emptyset$,
\item\label{it:3} $\frac{2}{3}\vol(B(y,\delta))<\vol(\bigcup_{p\in \mathcal{P}}B(p,\delta_p))$
and
\item\label{it:4} $\vol(\mathcal{B}^c\cap B(p,\delta_p))>\frac{1}{2}\vol(B(p,\delta_p))$.
\end{enumerate}
\end{problem}
In general and without any additional assumption, the answer of the problem is negative.  Actually, one can prove that
the action group $\Gamma$ is term ergodic with respect to volume measure if there is a family
$\{B(p,\delta_p)\}_{p\in \mathcal{P}}$ which satisfies in the properties (\ref{it:1})-(\ref{it:4}). Indeed, one can have
\begin{align*}
\vol(\mathcal{B}^c\cap B(y,\delta))&=
\vol[\bigcup_{p\in\mathcal{P}}(\mathcal{B}^c\cap B(p,\delta_p))] \\
&=\sum_{p\in\mathcal{P}}\vol(\mathcal{B}^c\cap
B(p,\delta_p))\\
&>\sum_{p\in\mathcal{P}}\frac{1}{2}\vol(B(p,\delta_p))\\
&=\frac{1}{2}\vol(\cup_{p\in
\mathcal{P}}B(p,\delta_p))\\
&=\frac{1}{3}\vol(B(y,\delta)),
\end{align*}
which is a contradiction and so $\vol(\mathcal{B})\in\{0,1\}$.
Therefore, by Theorem \ref{thmA}, minimality of action group $<\mathcal{G}\bigcup\mathcal{G}^{-1}>^+$ implies term ergodicity of it which is lead to the following counterexample.
\begin{example}
In \cite{fu}, Furstenberg constructed an analytic diffeomorphism $T$ of tours which preserve Haar measure and is minimal but not ergodic. So, $\Gamma=<T,T^{-1}>^+$ is minimal action group which is not a term ergodic.
This is a contradiction.
\end{example}
\section*{Acknowledgments}

We thank A. Fakhari, A. J. Homburg,  M. Nassiri, A. Navas and F. Khosh-Ahang for
useful discussions and suggestions.
\def\cprime{$'$}

\end{document}